\documentclass{amsart}
\usepackage{amssymb, amsmath}

\newtheorem{theorem}{Theorem}
\newtheorem{lemma}[theorem]{Lemma}

\newtheorem{remark}[theorem]{Remark}

\newcommand{\sign}{\operatorname{Sign}}
\newcommand{\CottonTres}{\operatorname{C}}
\newcommand{\CottonDos}{\operatorname{\tilde{C}}}
\newcommand{\CottonOp}{\operatorname{\hat{C}}}

\begin{document}

\title[Conformally symmetric three-manifolds]
{Three-dimensional conformally symmetric manifolds}
\author{\quad E. Calvi\~{n}o-Louzao, E. Garc\'{\i}a-R\'{\i}o, J. Seoane-Bascoy, R. V\'{a}zquez-Lorenzo}
\address{Department of Geometry and Topology, Faculty of Mathematics,
University of Santiago de Compostela, 15782 Santiago de Compostela, Spain}
\email{estebcl@edu.xunta.es, eduardo.garcia.rio@usc.es,javier.seoane@usc.es, ravazlor@edu.xunta.es}
\thanks{Supported by project MTM2009-07756 (Spain)}
\date{}

\begin{abstract}
The non-existence of non-trivial conformally symmetric manifolds in the three-dimensional  Riemannian setting is shown. In Lorentzian signature, a  complete local classification is obtained. Furthermore, the isometry classes are examined.
\end{abstract}

\maketitle

\section*{Introduction}

A pseudo-Riemannian manifold is said to be conformally symmetric if its Weyl tensor is parallel, i.e. $\nabla W=0$. It is known that any conformally symmetric Riemannian manifold is either locally symmetric (i.e., $\nabla R=0$) or locally conformally flat (i.e., $W=0$). In the non-trivial case ($\nabla W=0$ and $\nabla R\neq 0$, $W\neq 0$), the manifold $(M,g)$ is said to be \emph{essentially conformally symmetric}. The local and global geometry of essentially conformally symmetric pseudo-Riemannian manifolds has been extensively investigated by Derdzinski and Roter in a series of papers (see \cite{Derdzinski-Roter07,Derdzinski-Roter09} and the references therein for further information). It is worth emphasizing here that since the Weyl tensor vanishes in dimension three, conformally symmetric manifolds have been investigated only in dimension greater than four.
The main goal of this paper is to extend the study of conformal symmetric manifolds to the three-dimensional setting, where all the conformal information is codified by the Cotton tensor.

Let $\rho$ and $\tau$ denote the Ricci tensor and the scalar curvature of $(M,g)$.
Considering the Schouten tensor given by $S_{ij}=\rho_{ij}-\frac{\tau}{2(n-1)}g_{ij}$ where $n=\operatorname{dim}\, M$,  the \emph{Cotton tensor},
$\CottonTres_{ijk}=(\nabla_i S)_{jk}-(\nabla_j S)_{ik}$,
measures the failure of the Schouten tensor  to be a Codazzi tensor (see \cite{York71}).
It is well-known that any locally conformally flat manifold has vanishing Cotton tensor and the converse is also true in dimension $n=3$. Moreover, the Cotton tensor plays an important role in  Riemannian and pseudo-Riemannian geometry. The study of gradient Ricci solitons in locally conformally flat manifolds \cite{Cao-Chen12} or the Goldberg-Sachs theorem in pseudo-Riemannian geometry \cite{Gover-Hill-Nurowski11} are examples where the Cotton tensor appears naturally.

As a matter of notation, we say that a three-dimensional pseudo-Riemannian manifold is \emph{essentially conformally symmetric} if the Cotton tensor is parallel but the manifold is not locally conformally flat, i.e. $\nabla \CottonTres=0$ with $\CottonTres\neq0$. (Note that any three-dimensional locally symmetric manifold is locally conformally flat). Further observe that while conformally symmetry involves third-order derivatives of the metric in dimension $n\geq 4$, it is a fourth-order condition in dimension three.

The main result in the present paper is the following local description of essentially conformally symmetric three-dimensional manifolds.

\begin{theorem}\label{th:main}
A three-dimensional pseudo-Riemannian manifold is essentially conformally symmetric if and only if it is a strict Lorentzian Walker manifold, locally isometric to $(\mathbb{R}^3,(t,x,y),g_\mathfrak{a})$, where
\begin{equation}\label{eq:metric main theorem}
g_\mathfrak{a}=dt dy +dx^2+ (x^3+ \mathfrak{a}(y) x) dy^2,
\end{equation}
for an arbitrary smooth function $\mathfrak{a}(y)$.
\end{theorem}

\section{Proof of Theorem \ref{th:main}}\label{se:Proof}

We start with some previous lemmas, firstly considering the case when the manifold $(M,g)$ decomposes as a product.

\begin{lemma}\label{le:product metric}
A three-dimensional pseudo-Riemannian product manifold $(M=\mathbb{R}\times N,g={\pm} dt^2+g_N)$  is never essentially  conformally symmetric.
\end{lemma}
\begin{proof}
Let $(a,x)$ and $(b,y)$ be vector fields on $(M,g)$. Hence, the Schouten tensor $S$ satisfies
\[
\begin{array}{lcl}
S((a,x),(b,y)) & = & \rho((a,x),(b,y))-\frac{\tau}{4}g((a,x),(b,y))\\
\noalign{\medskip}
& = & \frac{\tau}{2}g(x,y)-\frac{\tau}{4}({\pm}ab+g(x,y))\\
\noalign{\medskip}
& = & \frac{\tau}{4}({\mp}ab+g(x,y))\\
\noalign{\medskip}
& = & \frac{\tau}{4}({\mp}dt^2+g_N)((a,x),(b,y)).\\
\end{array}
\]
Next observe that $g$ and
$h={\mp}dt^2+g_N$ are metrics on $M$ sharing the same Levi-Civita connection, and hence a straightforward calculation shows that
$\nabla S =\frac{1}{4}d\tau\otimes h$. As a consequence, the $(0,3)$-Cotton tensor is given by
\[
\CottonTres_{\alpha\beta\gamma}=(\nabla_\alpha S)_{\beta \gamma}-(\nabla_\beta S)_{\alpha\gamma}=\frac{1}{4}\left(d\tau(\alpha)h_{\beta\gamma}-d\tau(\beta)h_{\alpha\gamma}\right),
\]
from where we get
\[
\nabla_\mu \CottonTres_{\alpha\beta\gamma}=\dfrac{1}{4}\left(\left(\operatorname{Hes}_\tau\right)_{\mu\alpha}h_{\beta\gamma}-\left(\operatorname{Hes}_\tau\right)_{\mu\beta}h_{\alpha\gamma}  \right),
\]
where $\operatorname{Hes}_\tau$ denotes the Hessian of the scalar curvature,
$\operatorname{Hes}_\tau(X,Y)=X(Y \tau)-(\nabla_XY)\tau$.
Then, $\nabla_i\CottonTres_{tjt}=\frac{1}{4}\left(\operatorname{Hes}_\tau\right)_{ij}$ and therefore the product manifold is conformally symmetric if and only if $\operatorname{Hes}_\tau$ vanishes, i.e., $\nabla \tau$ is a parallel vector field on $N$.
Finally, if $\nabla\tau\neq 0$ then $N$ splits locally as  $N\equiv\mathbb{R}\times\mathbb{R}$, while if $\nabla \tau= 0$ then $N$ has constant curvature. In any case one has that $\nabla \CottonTres=0$ implies $\CottonTres=0$ thus showing that $(M,g)$ is locally conformally flat, which finishes the proof.
\end{proof}

Next lemma shows that any essentially conformally symmetric three-dimensional manifold is locally indecomposable but not irreducible, i.e., it admits a parallel degenerate line field. These manifolds have been extensively investigated in the literature and usually referred to as \emph{Brinkmann waves} or \emph{Walker manifolds} (see for example \cite{Brozos-Garcia-Gilkey-Nikevic-Vazquez09} and the references therein).

\begin{lemma}\label{le:algebraic classification}
Any three-dimensional essentially conformally symmetric manifold $(M,g)$  is a Walker manifold.
\end{lemma}

\begin{proof}
In what follows, we explicitly use that $\operatorname{dim}\, M=3$ to associate a $(0,2)$-tensor field to the usual $(0,3)$-Cotton tensor.
Let $\star:\Lambda^p(M)\rightarrow\Lambda^{3-p}(M)$ denote the Hodge $\star$-operator and consider the
Cotton $2$-form $C_i= \frac{1}{2}\CottonTres_{nmi}dx^n\wedge dx^m$.
Using the Hodge $\star$-operator, the $(0,2)$-Cotton tensor is associated to the Cotton
$2$-form  by $\star
C_i=\frac{1}{2}\CottonTres_{nmi}\epsilon^{nm\ell}dx^\ell$, thus resulting
\[
\CottonDos_{ij}=\frac{1}{2\sqrt{g}}\CottonTres_{nmi}\epsilon^{nm\ell}g_{\ell j},
\]
where $\epsilon^{123}=1$. Moreover, associated to the $(0,2)$-Cotton tensor, the Cotton operator is defined by $\CottonDos(x,y)=g(\CottonOp(x),y)$. Since the metric tensor and the Hodge $\star$-operator are parallel, the conformal symmetry of any three-dimensional manifold can be equivalently stated in terms of  the $(0,2)$-Cotton tensor $\CottonDos$, or in terms of the Cotton operator~$\CottonOp$.

Next we consider three-dimensional manifolds with parallel Cotton operator and analyze the different possibilities for the Jordan normal form of $\CottonOp$.
Assume that the Cotton operator diagonalizes.
If the Cotton operator is parallel, then the eigenvalues of $\CottonOp$ are constant and the corresponding eigenspaces define parallel distributions on $M$.
Since the Cotton tensor is traceless it has at least two different eigenvalues, unless the manifold is locally conformally flat. In any case, $\CottonOp$ always has a distinguished eigenvalue of multiplicity one and thus $(M,g)$ admits locally a de Rham decomposition as a product $(\mathbb{R}\times N, \pm dt^2+g_N)$, where $N$ is a surface. Then Lemma \ref{le:product metric} shows that $(M,g)$ is locally conformally flat.

Assume that the Cotton operator has a complex eigenvalue. Then, with respect to an orthonormal local frame $\{e_1,e_2,e_3\}$ of signature $(++-)$ one has
\[
\CottonOp=\left(
\begin{array}{ccc}
\lambda & 0 & 0\\
0 & \alpha & \beta\\
0 & -\beta & \alpha
\end{array}
\right).
\]
Since the  Cotton operator   is parallel, the distribution defined by the eigenspace corresponding to the eigenvalue $\lambda$ is parallel. Such distribution is spacelike and hence the manifold decomposes locally as a product $\mathbb{R}\times N$, where $N$ is a surface. Then,   Lemma \ref{le:product metric} implies  that $(M,g)$ is locally conformally flat.

If the minimal polynomial of the Cotton operator $\CottonOp$  has a root of multiplicity two, then there exists a local frame $\{e_1,e_2,e_3\}$, with $g(e_1,e_1)=g(e_2,e_3)=1$ such that $\CottonOp$ expresses with respect to that frame as
\[
\CottonOp=\left(
\begin{array}{ccc}
\lambda & 0 & 0\\
0 & \alpha & 1\\
0 & 0 & \alpha
\end{array}
\right).
\]
Suppose first that $\lambda\neq 0$; in this case, since  $\CottonOp$ is parallel the  spacelike distribution defined by the eigenspace corresponding to the eigenvalue $\lambda$ is also parallel. Then the manifold decomposes locally as a product and Lemma \ref{le:product metric} shows that $(M,g)$ is locally conformally flat. Next assume $\lambda=0$; since the Cotton operator is trace-free it must be $2$-step nilpotent (i.e., $\CottonOp^2=0$, $\CottonOp\neq 0$). Then, $\operatorname{Im}(\CottonOp)=\langle \CottonOp(e_3)\rangle=\langle e_2\rangle$. Moreover, taking into account that the Cotton operator is parallel, we have
\[
0=(\nabla_X \CottonOp)(e_3)=\nabla_X (\CottonOp(e_3))-\CottonOp(\nabla_X e_3),
\]
from where $\CottonOp(\nabla_X e_3)=\nabla_X (\CottonOp(e_3))$ and hence $\operatorname{Im}(\CottonOp)$ is a null and parallel one-dimensional distribution on $(M,g)$, from where one concludes that $g$ is a Walker metric.

Finally, consider the case when the minimal polynomial of the Cotton operator has a root of multiplicity $3$. Since the Cotton operator is trace-free    it  must be  $3$-step nilpotent. Proceeding as before and using the fact that the Cotton operator  is parallel one easily shows that    $\operatorname{Ker}(\CottonOp)$ is a null and parallel distribution  and thus $(M,g)$  is  a Walker manifold.
\end{proof}

\begin{remark}\rm
In the Riemannian case the Cotton operator diagonalizes and thus, the non-existence of three-dimensional essentially conformally symmetric Riemannian manifolds follows from the proof of Lemma \ref{le:algebraic classification}.
\end{remark}

\begin{proof}[Proof of Theorem \ref{th:main}]
Let $(M,g)$ be a three-dimensional essentially conformally symmetric Lorentzian manifold. It follows from
Lemma  \ref{le:algebraic classification} that $(M,g)$ is indecomposable but not irreducible, and hence a Walker manifold.

Three-dimensional Walker manifolds admit local coordinates $(t,x,y)$ where the metric expresses as
(see \cite{Brozos-Garcia-Gilkey-Nikevic-Vazquez09} and the references therein)
\begin{equation}\label{eq:Walker metric}
g= dt dy + dx^2 + f(t,x,y) dy^2,
\end{equation}
for some smooth function $f(t,x,y)$.
In the special case when the parallel degenerate line field is spanned by a parallel null vector field, the coordinates above can be further specialize so that the metric takes the form \eqref{eq:Walker metric} for some function  $f(x,y)$.

Then the Levi-Civita connection  is determined (up to the usual symmetries) by
\begin{equation}\label{eq:Levi Civita}
\nabla_{\partial_t}\partial_y=\tfrac{1}{2}f_t\partial_t,\quad \nabla_{\partial_x}\partial_y=\tfrac{1}{2}f_x\partial_t,\quad
\nabla_{\partial_y}\partial_y=\tfrac{1}{2} \left(f_y+f f_t\right)\partial_t-\tfrac{1}{2}f_x \partial_x -\tfrac{1}{2}f_t \partial_y,
\end{equation}
and the Ricci tensor is given by
\begin{equation}\label{eq:Ricci tensor}
\rho(\partial_t,\partial_y)=\tfrac{1}{2}f_{tt},\quad \rho(\partial_x,\partial_y)=\tfrac{1}{2}f_{tx}, \quad
\rho(\partial_y,\partial_y)=\tfrac{1}{2}\left(f f_{tt}-f_{xx}\right).
\end{equation}
Moreover, the   $(0,2)$-Cotton tensor of a Walker metric (\ref{eq:Walker metric}) is characterized by
\begin{equation}\label{eq:Cotton tensor}
\begin{array}{l}
\CottonDos(\partial_t,\partial_x)=-\frac{1}{4} f_{ttt},\quad \CottonDos(\partial_t,\partial_y)=\frac{1}{4} f_{ttx}, \quad
\CottonDos(\partial_x,\partial_x)=-\frac{1}{2}  f_{ttx},
\\
\noalign{\medskip}
\CottonDos(\partial_x,\partial_y)=\frac{1}{4} \left(2 f_{txx}+f_{tty}-f f_{ttt}\right),
 \\\noalign{\medskip}
\CottonDos(\partial_y,\partial_y)=\frac{1}{4}\left(  f_x f_{tt}-2
   f_{xxx}-f_t f_{tx}-2 f_{txy}+2 f f_{ttx}\right).
\end{array}
\end{equation}
A long but straightforward calculation shows that a Walker metric (\ref{eq:Walker metric}) has parallel Cotton tensor if and only if
\begin{equation}\label{eq:Conditions Walker}
\left\{
\begin{array}{rll}
f_{tttt}=f_{tttx}=f_{ttxx}=f_{txxx}&=&0,
\\\noalign{\medskip}
f_t f_{ttt}-2 f_{ttty}&=&0,
\\\noalign{\medskip}
 2 f_{ttxy}-f_x f_{ttt}&=&0,
 \\[0.05in]
4 f_{txxy}+ \left(2 f_{txx}+ f_{tty}\right)f_t+2 f_{ttyy}-3 f_xf_{ttx}-f_y f_{ttt}-2 f f_{ttty}&=&0,
\\\noalign{\medskip}
 (f_{tx})^2+2 f_{xxxx}+f_t f_{txx}+2 f_{txxy}-f_{xx} f_{tt}-2 f_x f_{ttx} &=&0,
 \\\noalign{\medskip}
 f_{tx} \left(f_t\right)^2+\left(2 f_{xxx}+3 f_{txy}\right) f_t+2 f_{xxxy}+f_{ty} f_{tx}&&\,
 \\\noalign{\medskip}
 + 2   f_{txyy}-f_{xy} f_{tt}- \left(2 f_{txx}+f_t f_{tt}+2
   f_{tty}\right)f_x-\left(f_y+f f_t\right) f_{ttx}&=&0.
\end{array}
\right.
\end{equation}
From the first equation in  (\ref{eq:Conditions Walker}) we get
\[f(t,x,y)=\alpha(x,y)+t \left(\beta(y)+x\, \xi(y)+x^2 \delta(y)\right)+t^2 (\mu(y)+x \,\phi(y))+t^3
   \gamma(y).\]
Now, differentiating the second equation in (\ref{eq:Conditions Walker}) twice  with respect to $t$, it follows that $\gamma(y)=0$. Then the third equation in (\ref{eq:Conditions Walker}) transforms into   $\phi'(y)=0$ and therefore  $\phi(y)=K$. We differentiate now the fourth equation in (\ref{eq:Conditions Walker}) twice with respect to $t$  to obtain  $K=0$. Moreover, differentiating again the fourth equation in (\ref{eq:Conditions Walker})   twice with respect to $x$  we get
\begin{equation}\label{eq: twice x}
\delta (y) \left(2 \delta (y)+\mu'(y)\right)=0,
\end{equation}
and  differentiating once again the fourth equation in (\ref{eq:Conditions Walker}), in this case with respect to $t$,  we obtain
\begin{equation}\label{eq: once t}
\mu(y) \left(2 \delta (y)+\mu'(y)\right)=0.
\end{equation}
Hence,   Equations (\ref{eq: twice x}) and (\ref{eq: once t}) imply that
$\delta(y)=-\frac{1}{2}\mu'(y)$.
At this point, the  only non-zero component of the Cotton tensor is given by
\[
\begin{array}{l}
\CottonDos(\partial_y,\partial_y) =
	\frac{1}{8} \{4 \mu(y) \alpha_x(x,y) - 4 \alpha_{xxx}(x,y)
	- x^3 \mu'(y)^2 + 3 x^2 \xi(y) \mu'(y)
   \\\noalign{\medskip}
   \phantom{\CottonDos(\partial_y,\partial_y) = \frac{1}{8} \{}
   - 2 \beta(y)  \left(\xi(y) -x \mu'(y)\right)
   -2 x \xi(y)^2+4 x \mu''(y)-4 \xi'(y) \}.
\end{array}
\]
Now, differentiating the last equation in (\ref{eq:Conditions Walker}) with respect to $t$ a straightforward calculation shows that
\[\mu(y) \CottonDos(\partial_y,\partial_y)=0.\]
If $\CottonDos(\partial_y,\partial_y)=0$ at some point then the Cotton tensor  vanishes everywhere since it is parallel, and therefore the manifold is locally conformally flat. Thus, we may assume that $\mu(y)=0$.  Now, the fifth equation in (\ref{eq:Conditions Walker}) transforms into
\[
2\alpha_{xxxx}(x,y)+\xi(y)=0,
\]
and therefore  $\alpha(x,y)=-\frac{1}{48}x^4\xi(y)^2+\mathcal{D}(y)x^3+\mathcal{C}(y)x^2+\mathcal{B}(y)x+\mathcal{A}(y)$. Now, a long but straightforward calculation shows that the unique non-zero component of the $(0,2)$-Cotton tensor is given by
\[
\CottonDos(\partial_y,\partial_y)= -\tfrac{1}{4} \left(12 \mathcal{D}(y)+\beta(y) \xi(y)+2
   \xi'(y)\right),
\]
and   differentiating  the last equation in (\ref{eq:Conditions Walker}) with respect to $x$  we get
\[
8 \xi(y) \CottonDos(\partial_y,\partial_y)=0.
\]
As before, if $\CottonDos(\partial_y,\partial_y)$ does not vanish identically, then one has $\xi(y)=0$. Thus, the  Cotton tensor is determined by $\CottonDos(\partial_y,\partial_y)=-3\mathcal{D}(y)$, which implies that  $\mathcal{D}(y)\neq 0$ everywhere unless $(M,g)$ is locally conformally flat; moreover,  the last equation in (\ref{eq:Conditions Walker}) reduces to $\mathcal{D}(y)\beta(y)+\mathcal{D}'(y)=0$, from where   $\beta(y)=-\frac{\mathcal{D}'(y)}{\mathcal{D}(y)}$.

At this point, the metric (\ref{eq:Walker metric}) is determined by
\[
f(t,x,y)=-\frac{\mathcal{D}'(y)}{\mathcal{D}(y)} t+\mathcal{D}(y)x^3+\mathcal{C}(y)x^2+\mathcal{B}(y) x+\mathcal{A}(y).
\]
A straightforward calculation shows that  the Ricci operator of this metric is $2$-step nilpotent. A three-dimensional Walker manifold with $2$-step nilpotent Ricci operator admits a null and parallel vector field \cite{Brozos-Garcia-Gilkey-Nikevic-Vazquez09} and therefore the Walker coordinates $(t,x,y)$ can be specialized so that the metric expresses as
\begin{equation}\label{eq:strict walker}
	g_f=dtdy+dx^2+f(x,y)dy^2.
\end{equation}
Now, the only non-zero component of the Cotton tensor is
\[\CottonDos(\partial_y,\partial_y)=-\tfrac{1}{2}f_{xxx},\]
and  the non-vanishing components of $\nabla \CottonDos$ are given by
\[
	(\nabla_{\partial_x} \CottonDos)(\partial_y,\partial_y)=-\tfrac{1}{2}f_{xxxx}, \qquad (\nabla_{\partial_y} \CottonDos)(\partial_y,\partial_y)=-\tfrac{1}{2}f_{xxxy}.
\]
A direct calculation  shows that a strict Walker metric is essentially conformally symmetric if and only if
\begin{equation}\label{eq: Walker conformally symmetric}
f(x,y)=\kappa x^3+x^2 \mathcal{A}(y)+\mathcal{B}(y) x+\mathcal{C}(y),
\end{equation}
for arbitrary smooth functions $\mathcal{A}$, $\mathcal{B}$ and $\mathcal{C}$ and a non-zero real constant $\kappa$.

In what remains of the proof we show that any metric \eqref{eq: Walker conformally symmetric} is locally isometric to some metric \eqref{eq:metric main theorem} for a suitable function $\mathfrak{a}(y)$.
We proceed as in \cite{Garcia-Gilkey-Nikcevic}. Let $g_f$ be a Walker metric defined by
 \eqref{eq:strict walker} and consider the application:
\[
T(t,x,y)=(t-\phi_y x+ \psi,x+\phi,y),
\]
where $\phi$ and $\psi$ are smooth functions on $y$.  Then $T$ defines an isometry between $g_f$ and another Walker metric $g_{\tilde{f}}$ given by \eqref{eq:strict walker} for some function
\[
\tilde{f}(x,y)=f(x+\phi,y)-2x \phi_{yy}+\phi_y^2+2\psi_y\,.
\]
Now consider a Walker metric  $g_{\mathfrak{b},\kappa}=dt dy +dx^2+(\kappa x^3+\mathfrak{b}(y)x) dy^2$ defined by some arbitrary smooth function $\mathfrak{b}(y)$ and some non-zero constant $\kappa$. Then $T$ defines an isometry between $g_{\mathfrak{b},\kappa}$ and a Walker metric $g_{\tilde{f}}$ where
\[
\begin{array}{rcl}
\tilde{f}(x,y)&=&\kappa (x+\phi)^3+\mathfrak{b}(y)(x+\phi)-2x \phi_{yy}+\phi_y^2+2\psi_y\\
\noalign{\medskip}
&=&\kappa x^3+3\kappa \phi x^2 +(\mathfrak{b}(y)+3 k \phi^2-2\phi_{yy})x+\mathfrak{b}(y)\phi+\kappa\phi^2+\phi_y^2+2\psi_y.
\end{array}
\]
Setting $\mathcal{A}=3\kappa\phi$, $\mathcal{B}$ defined by $\mathcal{B}=\mathfrak{b}+3 k \phi^2-2\phi_{yy}$ and choosing $\psi$ so that $\mathcal{C}=\mathfrak{b}(y)\phi+\kappa\phi^2+\phi_y^2+2\psi_y$,
one has that $T$ defines an isometry between
$g_{\mathfrak{b},\kappa}$ and a Walker metric $g_f$ with $f(x,y)$ given by Equation (\ref{eq: Walker conformally symmetric}).

Finally observe that $\kappa$ is not relevant in the previous discussion since
\[
\tilde{T}(t,x,y)=\left(\sqrt{\varepsilon \kappa}t,\varepsilon x,\dfrac{1}{\sqrt{\varepsilon\kappa}}y\right)
\]
is an isometry between $g_{\mathfrak{b},\kappa}$ and
$g_\mathfrak{a}=dt dy +dx^2+(x^3+\mathfrak{a}(y)x)dy^2$,
where the function $\mathfrak{a}(y)$ is given by $\mathfrak{a}(y)=(|\kappa|)^{-1}\mathfrak{b}(|\kappa|^{-1/2}y)$ and $\varepsilon=\sign(\kappa)$.
This concludes the proof.
\end{proof}

\begin{remark}\rm
Any essentially conformally symmetric metric $g_{\mathfrak{a}}$ given by (\ref{eq:metric main theorem}) is defined by a function $\mathfrak{a}(y)$. Hence, it is natural to consider whether two different functions $\mathfrak{a}(y)$ and $\mathfrak{b}(y)$ determine the same isometry class. In answering this question, first of all observe that the kernel and the image of the Ricci operator $\hat{\rho}$ (defined by $\rho(X,Y)=g(\hat{\rho}X,Y)$) of any metric given by (\ref{eq:metric main theorem}) are generated by
\[
\operatorname{Ker}\hat{\rho}=\langle\{ \partial_t,\partial_x\}\rangle,\qquad \operatorname{Im}\hat{\rho}=\langle\{\partial_t\}\rangle.
\]
Therefore, any (local) isometry must preserve these subspaces.

Let $\Phi=({}^1\Phi(t,x,y),{}^2\Phi(t,x,y),{}^3\Phi(t,x,y))$ be an isometry between $g_\mathfrak{a}$ and $g_\mathfrak{b}$, i.e. $\Phi^\star g_\mathfrak{b}=g_\mathfrak{a}$. Since $\Phi$ has to preserve
$\operatorname{Ker}\hat{\rho}$ and $\operatorname{Im}\hat{\rho}$, it follows that ${}^3\Phi$ depends only on the coordinate $y$ and ${}^2\Phi$ is a function of  the coordinates $x$ and $y$. Moreover,
\[
1=g_\mathfrak{a}(\partial_x,\partial_x)=g_\mathfrak{b}(\Phi_\star \partial_x, \Phi_\star\partial_x)= \left({}^2\Phi_x\right)^2\,,
\]
and thus, ${}^2\Phi(x,y)=\varepsilon_1 x+ \varphi(y)$ with $\varepsilon_1^2=1$. Furthermore, any isometry has to preserve the Ricci tensor (i.e., $\Phi^\star \rho_\mathfrak{b}=\rho_\mathfrak{a}$, where $\rho_{\mathfrak{a}}$ and $\rho_{\mathfrak{b}}$ are the Ricci tensors of $g_\mathfrak{a}$ and $g_\mathfrak{b}$, respectively). Hence, at any point $p=(t,x,y)$
\[
-3x=\rho_{\mathfrak{a}}(\partial_y,\partial_y)_{|_p}
=\rho_{\mathfrak{b}}(\Phi_\star\partial_y,\Phi_\star\partial_y)_{|_{\Phi(p)}}=-3(\varepsilon_1 x+\varphi(y))\left({}^3\Phi_y\right)^2,
\]
from where it follows that  $\varphi(y)=0$, ${}^3\Phi(y)=\varepsilon_2 y+\alpha$ and $\varepsilon_1=1$  with $\varepsilon_2^2=1$ and $\alpha$ an arbitrary constant.
Now,
\[
0=g_\mathfrak{a}(\partial_x,\partial_y)=g_\mathfrak{b}(\Phi_\star\partial_x,\Phi_\star\partial_y)=\varepsilon_2\, {}^1\Phi_x,
\]
from where we obtain that ${}^1\Phi$ depends only on the coordinates $t$ and $y$, i.e. ${}^1\Phi(t,x,y)=\Psi(t,y)$.
In addition,
\[
1=g_\mathfrak{a}(\partial_t,\partial_y)=g_\mathfrak{b}(\Phi_\star\partial_t,\Phi_\star\partial_y)=\varepsilon_2\, \Psi_t.
\]
Then, $\Psi(t,y)=\varepsilon_2 t+\Upsilon(y)$. So, for any point $p$ we have
\[
x^3+\mathfrak{a}(y)x=g_\mathfrak{a}(\partial_y,\partial_y)_{|_p}
=g_\mathfrak{b}(\Phi_\star\partial_y,\Phi_\star\partial_y)_{|_{\Phi(p)}}=x^3+\mathfrak{b}(\varepsilon_2y+\alpha)x+2\varepsilon_2\Upsilon_y.
\]
Hence, $\Upsilon(y)=\beta$. Finally, $\Phi$ is an isometry
from $g_\mathfrak{a}$ to $g_\mathfrak{b}$ if and only if
\[
\Phi=(\varepsilon_2\,t+\beta,x,\varepsilon_2 y+\alpha), \, \text{and}\quad \mathfrak{a}(y)=\mathfrak{b}(\varepsilon_2 y +\alpha), \quad \varepsilon_2^2=1.
\]
This shows that \emph{the moduli space of isometry classes of essentially conformally symmetric three-dimensional manifolds coincides with the space of smooth functions of one-variable $\mathfrak{a}(y)$,  up to constant speed parametrization}.
\end{remark}

\begin{remark}\rm
Any essentially conformally symmetric Lorentzian manifold of  dimension $n\geq 4$ has  recurrent Ricci curvature \cite{Derdzinski-Roter78}.  This behavior also holds in dimension $n=3$ since it can be easily shown that  any  Walker metric  (\ref{eq:metric main theorem}) (indeed, any Walker metric given by \eqref{eq:strict walker}) has  recurrent Ricci curvature.

Clearly  any three-dimensional $2$-symmetric manifold (i.e., $\nabla^2 R=0$ but $\nabla R\neq 0$) has parallel Cotton tensor. From \cite{Alekseevsky-Galaev11, Blanco-Sanchez-Senovilla} it is easy to show that any three-dimensional $2$-symmetric manifold is locally conformally flat. Therefore, a three-dimensional essentially conformally symmetric manifold cannot be  $2$-symmetric.

It follows from the work in  \cite{Garcia-Gilkey-Nikcevic} that essentially conformally symmetric manifolds of dimension three are not locally homogeneous (even more, they cannot be 1-curvature homogeneous).
\end{remark}

\section{Geometric solitons}\label{se:Remarks}

Finally we examine the role of essentially conformally symmetric three-manifolds in the construction of solitons for different geometric evolution equations. In all cases discussed below, solitons correspond to generalized fixed points (i.e., fixed points up to homotheties and diffeomorphisms) of the corresponding flows. Therefore, geometric solitons provide distinguished metrics for the different geometric objects under consideration.

\subsection*{Cotton solitons}
The Cotton flow is a geometric flow associated to the Cotton tensor, which is given by a one-parameter family of metrics $g(t)$ satisfying the equation
$\frac{\partial}{\partial t} g(t)=\mu \CottonDos_{g(t)}$, where $\mu$ is a real constant.
A three-dimensional pseudo-Riemannian manifold is said to be  a  \emph{Cotton soliton} if there exists a vector field $X$ such that $\mathcal{L}_Xg+\CottonDos=\lambda g$ for some real constant $\lambda$, where $\mathcal{L}$ denotes the Lie derivative. The soliton is shrinking, steady or expanding according to $\lambda>0$, $\lambda=0$ or $\lambda<0$, and is said to be a \emph{gradient Cotton soliton} if the Cotton soliton vector field $X$ is the gradient of a suitable potential function, $X=\nabla\varphi$. See \cite{Calvino-Garcia-Vazquez12} and the references therein for more information on Cotton solitons.

It follows from the work in \cite{Calvino-Garcia-Vazquez12} that any three-dimensional essentially  conformally symmetric pseudo-Riemannian manifold is a steady gradient Cotton soliton such that the gradient of the potential function $\varphi$, $\nabla\varphi$, is a null vector field.

Moreover, the existence of other kinds (shrinking or expanding) of Cotton solitons depends on the existence of non-Killing homothetic vector fields on $(M,g)$. Indeed, two Cotton soliton vector fields $X_1$ and $X_2$ differ in a homothetic vector field since
\[
\mathcal{L}_{X_1-X_2}g=\mathcal{L}_{X_1}g-\mathcal{L}_{X_2}g=\lambda_1 g-\CottonDos-\lambda_2 g+\CottonDos
=(\lambda_1-\lambda_2)g,
\]
and hence $(M,g,X_1)$ and $(M,g,X_2)$ are two distinct Cotton solitons if and only if $\xi=X_1-X_2$ is a homothetic vector field. Homothetic vector fields are strongly related with self-similar solutions of the Yamabe flow, specially when the scalar curvature vanishes. In such a case, homothetic vector fields and Yamabe solitons coincide (see \cite{Calvino-Seoane-Vazquez-Vazquez} for further information in Lorentzian Yamabe solitons).

A straightforward calculation (that we omit for sake of brevity) shows that a three-dimensional essentially conformally symmetric manifold $(M,g)$ admits a homothetic vector field (i.e., a vector field $X$ satisfying $\mathcal{L}_Xg=\lambda g$ for some constant~$\lambda$) if and only if the metric \eqref{eq:metric main theorem} is given by a function $\mathfrak{a}(y)$ satisfying
\[
\mathfrak{a}(y)=\frac{\alpha }{(4 \beta -\lambda  y)^4},
\]
for arbitrary constants $\alpha$, $\beta$ and $\lambda$. Moreover the homothetic vector field $\xi$ is given by
\[
\xi(t,x,y)=\left(\frac{5 \lambda t}{4}+\kappa,\frac{\lambda x}{2},\beta-\frac{\lambda  y}{4}\right).
\]

In this case, the metric admits expanding and shrinking Cotton solitons, depending on the sign of $\lambda$, given by the vector
\[
X(t,x,y)=\left(\tilde\kappa+\dfrac{5 \lambda  t}{4}+\frac{3 y}{2},\dfrac{\lambda x}{2},\beta-\dfrac{\lambda y}{4}\right)\,.
\]

\subsection*{Ricci solitons}
A vector field $X$ is said to be a \emph{Ricci soliton} vector field if and only if
$\mathcal{L}_X g+\rho=\lambda g$, and $(M,g,X)$ is said to be a Ricci soliton which is named to be shrinking, steady or expanding according to $\lambda>0$, $\lambda=0$ or $\lambda<0$, respectively.
The Ricci soliton is said to be a gradient Ricci soliton if the Ricci soliton vector field is the gradient of a suitable potential function, $X=\nabla\psi$.

Essentially conformally symmetric three-dimensional manifolds do not admit any gradient Ricci soliton (see for example \cite{Garcia-Gilkey-Nikcevic} and the references therein). However they admit non-gradient Ricci soliton structures in some cases.
In \cite{Brozos-Calvaruso-Garcia-Gavino} the authors study when a strict Walker metric $g_f$ admits a Ricci soliton. They obtain that any strict Walker metric admits such a soliton if and only if the vector field $X$ is given by
\[
X(t,x,y)=\left(t(\lambda-\beta)-x \omega'(y)+\mu(y),\dfrac{1}{2}\lambda x+\omega(y),\beta y+\gamma\right)\,,
\]
for some real constants $\beta$, $\gamma$ and smooth functions $\omega(y)$ and $\mu(y)$ satisfying the partial differential equation
\begin{equation}\label{eq:Ricci soliton Walker metric}
2\beta f-\lambda f+2\mu'(y)-2 x \omega''(y)+f_y(\beta y+\gamma)+f_x(\dfrac{\lambda}{2}x+\omega(y))-\dfrac{1}{2}f_{xx}=0\,.
\end{equation}
Specializing those results for essentially conformally symmetric manifolds as in Theorem \ref{th:main}, Equation \eqref{eq:Ricci soliton Walker metric} becomes
\begin{equation}
\mathfrak{a}'(y) \left(\gamma -\frac{\lambda  y}{4}\right)-\lambda
   \mathfrak{a}(y)-3=0.
\end{equation}
Hence, an essentially conformally symmetric three-manifold is a Ricci soliton if and only if the function $\mathfrak{a}(y)$ in Equation \eqref{eq:metric main theorem} is of the form
\[
\mathfrak{a}(y)=\left\{\begin{array}{lll}
\frac{\alpha }{(4 \gamma -\lambda  y)^4}-\frac{3}{\lambda },& \text{if} &\lambda\neq 0,\\
\noalign{\medskip}
\dfrac{3}{\gamma}y+\alpha, & \text{if}& \lambda=0,
\end{array}
\right.
\]
where $\alpha$ is an arbitrary constant, and the Ricci soliton vector field takes the form
\[
X(t,x,y)=\left(\frac{5\lambda t}{4}+\kappa,\frac{\lambda x}{2},\gamma-\frac{\lambda y}{4}\right)\,,
\]
where $\kappa$ is an arbitrary real constant.

\medskip



\begin{thebibliography}{99}
\bibitem{Alekseevsky-Galaev11}
D.V. Alekseevsky and A. Galaev, Two-symmetric Lorentzian manifolds, \emph{J. Geom. Phys.}  \textbf{61} (2011),  2331--2340.

\bibitem{Blanco-Sanchez-Senovilla}
O. F. Blanco, M. S\'{a}nchez and J. M. Senovilla, Structure of second-order symmetric
lorentzian manifolds, \emph{J. Eur. Math. Soc.}, to appear.

\bibitem{Brozos-Calvaruso-Garcia-Gavino} M. Brozos-V\'{a}zquez, G. Calvaruso, E. Garc\'{\i}a-R\'{\i}o and S. Gavino-Fern\'{a}ndez,
    Three-dimensional Lorentzian homogeneous Ricci solitons,
    \emph{Israel J. Math} \textbf{188} (2012), 385--403.

\bibitem{Brozos-Garcia-Gilkey-Nikevic-Vazquez09}
M. Brozos-V\'{a}zquez, E. Garc\'{i}a-R\'{i}o, P. Gilkey, S.
Nik\v{c}evi\'{c} and  R. V\'{a}zquez-Lorenzo, \emph{The geometry of Walker
manifolds}, Synthesis Lectures on Mathematics and Statistics
\textbf{5}, Morgan \& Claypool Publ., 2009.


\bibitem{Cao-Chen12}
H-D. Cao and Q. Chen, On locally conformally flat gradient steady Ricci solitons, \emph{Trans. Amer. Math. Soc.}  \textbf{364} (2012),  2377--2391.



\bibitem{Calvino-Garcia-Vazquez12}
E. Calvi\~{n}o-Louzao, E. Garc\'{\i}a-R\'{\i}o and R. V\'{a}zquez-Lorenzo, A note on compact Cotton solitons, \emph{Classical Quantum Gravity}  \textbf{29} (2012),  205014 (5pp).




\bibitem{Calvino-Seoane-Vazquez-Vazquez}
E. Calvi\~{n}o-Louzao, J. Seoane-Bascoy, M.E. V\'{a}zquez-Abal and R. V\'{a}zquez-Lorenzo, Three-dimensional homogeneous Lorentzian Yamabe solitons, \emph{Abh. Math. Semin. Univ. Hambg.}  \textbf{82} (2012), 193--203.



\bibitem{Derdzinski-Roter77}
A. Derdzinski and W. Roter, On conformally symmetric manifolds with metrics of indices $0$ and $1$, \emph{Tensor (N.S.)}
\textbf{31} (1977),  255--259.

\bibitem{Derdzinski-Roter78}
A. Derdzinski and W. Roter, Some theorems on conformally symmetric manifolds, \emph{Tensor (N.S.)}
\textbf{32} (1978),  11--23.



\bibitem{Derdzinski-Roter07}
A. Derdzinski and W. Roter, Projectively flat surfaces, null parallel distributions, and conformally symmetric manifolds, \emph{Tohoku Math. J.}  \textbf{59} (2007),  565--602.

\bibitem{Derdzinski-Roter09}
A. Derdzinski and W. Roter, The local structure of conformally symmetric manifolds, \emph{Bull. Belg. Math. Soc. Simon Stevin}  \textbf{16}
(2009),  117--128.


\bibitem{Garcia-Gilkey-Nikcevic}
E. Garc\'{\i}a-R\'{\i}o, P. Gilkey and S. Nik$\check{\text{c}}$evi\'c, Homogeneity of Lorentzian three-manifolds with recurrent curvature, arXiv:1210.7764v2.


\bibitem{Gover-Hill-Nurowski11}
A. Rod Gover, C. Denson Hill and P. Nurowski, Sharp version of the Goldberg-Sachs theorem, \emph{Ann. Math. Pura Appl.}
\textbf{190} (2011),  295--340.






\bibitem{York71} J. W. York, Jr., Gravitational degrees of freedom and the initial-value problem,
    \emph{Phys. Rev. Lett.} \textbf{26} (1971), 1656--1658.

\end{thebibliography}
\end{document}